\documentclass[11pt]{article}
\usepackage[margin=0.1in,footskip=0.2in]{geometry}
\usepackage[english]{babel}
\usepackage{amssymb}
\usepackage{amsthm}
\usepackage{amssymb}
\usepackage{setspace}
\usepackage{amsmath,amsthm,amscd,amsfonts,mathrsfs,amssymb}
\usepackage{graphicx,enumerate}
\usepackage{mathptmx}

\usepackage[all]{xy}
\usepackage{fullpage} 
\usepackage{color}
\usepackage{makeidx}
\usepackage{alltt}
\usepackage{setspace}
\usepackage{hyperref}

\newtheorem{theorem}{Theorem}[section]
\newtheorem{lemma}[theorem]{Lemma}

\newtheorem{proposition}[theorem]{Proposition}
\theoremstyle{definition}
\newtheorem{remark}[theorem]{Remark}
\theoremstyle{definition}

\newcommand{\dd}{{\mathrm{d}}}
\newcommand{\ssl}[2]{\mathrm{SL}_{#1}(\mathbb{#2})}
\newcommand{\ggl}[2]{\mathrm{GL}_{#1}(\mathbb{#2})}
\newcommand{\oo}[2]{\mathrm{O}_{#1}(\mathbb{#2})}
\newcommand{\GL}[1]{\mathrm{GL}_{#1}}
\newcommand{\pra}[1]{\left(#1\right)}

\newcommand{\sign}{\operatorname{sign}}

\newcommand{\blue}{\textcolor{blue}}

\renewcommand{\mod}{\textrm{mod }}

\newcommand{\WJ}{\mathcal{W}_{\textrm{Jacquet}}}
\newcommand{\spaces}{ { } \quad\; { } \quad \;  { } \quad\; { } \quad \;  { } \quad    }
\newcommand{\nothing}[1]{{}}
\newcommand{\gammafactor}[1]{\Gamma\pra{\tfrac{#1-\alpha}{2}} \Gamma\pra{\tfrac{#1-\beta}{2}}  \Gamma\pra{\tfrac{#1-\gamma}{2}}}
\newcommand{\gammafactorminus}[1]{\Gamma\pra{\tfrac{#1+\alpha}{2}} \Gamma\pra{\tfrac{#1+\beta}{2}}  \Gamma\pra{\tfrac{#1+\gamma}{2}}}

\newcommand{\levelN}{^{\scriptscriptstyle {(N )}}}
\newcommand{\chibs}{\overline{\chi^*}}
\newcommand{\chis}{\chi^*}
\newcommand{\chib}{{\bar{\chi}}}
\newcommand{\cs}{{c^*}}
\newcommand{\psib}{\bar{\psi}}

\newcommand{\ATF}{A_{\tilde{F}}}
\newcommand{\AF}{A_F}
\begin{document}

\title{\scshape The Voronoi formula on GL(3) with ramification}

\author{Fan Zhou}

\maketitle
\begin{abstract}
Firstly we prove that the Voronoi formula of Miller-Schmid type applies to automorphic forms on GL(3) for the congruence subgroup $\Gamma_0(N)$, when the conductor of the additive character in the formula is a multiple of $N$. 
As an application, we produce a result about the functional equation of $L$-function of the automorphic form on GL(3) twisted by Dirichlet characters. 
Secondly we prove that a similar formula applies to automorphic forms on GL(3) for the congruence subgroup $\Gamma_0(N)$, when the conductor of the additive character in the formula is coprime with $N$. 
\\
\\
MSC: 
11F30  (Primary), 11F55 
\\
\\
\end{abstract}
\tableofcontents
\section{Introduction}
A Voronoi formula is a Poisson-type summation formula involving the Fourier coefficients of an automorphic form, with the coefficients twisted by an additive character on either side. 
The Voronoi formula on $\GL{2}$ is a powerful standard tool to study automorphic forms on $\GL{2}$ and their $L$-functions.
The formula is explicity and implicit at several places (\cite{good, dukeiwaniec, jutila}).

Let $f$ be an automorphic form on $\GL{2}$ for $\Gamma_0(N)$ and a Dirichlet character $\psi: \mathbb{Z}/N\mathbb{Z}\to \mathbb{C}^\times$, with the $n^{\textrm{th}}$ Fourier coefficient $\lambda_f(n)$.
A Voronoi formula treats the expressions of the type
$$\sum_n \lambda_f(n)\exp\pra{2\pi i\frac{an}{c}}\omega(n),$$
where $(a,c)=1$ and $\omega$ is a test function on $\mathbb{R}$. 
There are two distinct scenarios for the relationship between $c$ and $N$: 
\begin{center}
$N|c\quad $  and $\quad (c, N)=1$. 
\end{center}
The first scenario is a formula 
$$\sum_n \lambda_f(n)\exp\pra{2\pi i\frac{an}{c}}\omega(n)=\frac{\psi(\bar a)}{c}\sum_n \lambda_f(n)\exp\pra{-2\pi i\frac{\bar an}{c}}\Omega\pra{\tfrac{n}{c^2}},$$ 
where $N|c$, $a\bar a\equiv 1\;\mod c$ and $\Omega$ is an integral transform of $\omega$. 
The second scenario is a formula 
$$\sum_n \lambda_f(n)\exp\pra{2\pi i\frac{an}{c}}\omega(n)=\frac{\psi (c)}{c\sqrt{N}}\sum_n \lambda_{\tilde{f}}(n)\exp\pra{-2\pi i\frac{\bar a\bar N n}{c}}\Omega\pra{\tfrac{n	}{c^2N}},$$ 
where $\tilde{ f}$ is the contragredient form of $f$ and $\bar N$ is a multiplicative inverse of $N\;\mod c$.
Both scenarios can be combined in \cite[Appendix]{kmv}, \cite{michel} and \cite{templier}  using the Atkin-Lehner-Li theory
and especially the Atkin-Li operator on $\GL 2$ from \cite{atkinli}. More precisely, in \cite{kmv}, at every finite place, either the first scenario or the second scenario happens, i.e., for every prime $p$ with $p^k||N$, either $p^k|c$ or $p\nmid c$. 

The Voronoi formulas in higher rank are more recent. The Voronoi formula for automorphic forms for $\ssl 3 Z$ was first discovered by Miller and Schmid in \cite{millerschmid1}. Other proofs and generalizations can be found in \cite{millerschmid2, ichinotemplier, goldfeldli1, goldfeldli2, zhou, kiralzhou, millerzhou}. In \cite{ichinotemplier}, a general formula is formulated by Ichino and Templier over number fields and 
it allows automorphic forms with ramification (instead of automorphic forms for $\ssl n Z$). The ramification is only allowed  at places disjoint from the additive character (the second scenario). 
There is also an interesting special case of Voronoi formula on GL(3) with ramification in \cite[Lemma 1.3]{buttcanekhan}, which is not covered by \cite{ichinotemplier}.
It is believed by \cite{millerschmid2} that a general formula can be achieved if some Atkin-Lehner information is available on $\GL{n}$ for $n\geq 3$.
So far no Atkin-Lehner theory is available on $\GL{n}$ with $n\geq 3$. 

In this paper, we firstly introduce a Voronoi formula of Miller-Schmid type (Theorem \ref{main}) for automorphic forms on $\GL{3}$ of level $N$ for the congruence subgroup $\Gamma_0(N)$ when 
the conductor of the additive character is a multiple of the level $N$
(i.e., $N|c$, analogous to the first scenario) without involving any Atkin-Lehner theory on $\GL{3}$. 
The case treated in this paper is the exact opposite of \cite{ichinotemplier}.
Let $F$ be a Maass form for $\Gamma_0(N)$ with a Dirichlet character modulo $N$.  
Let $A_F(\;,\;)$ be the Fourier-Whittaker coefficient and let $m$ be a nonzero integer. For $c$ with $N|c$ we have the formula
	\begin{align}\nonumber
	\sum_{n=1}^\infty &A_F(n,m)\exp\pra{2\pi i\frac{\bar an}{c}}\omega(n)
	\\
	&=c\psi(a)\sum_{m_2\neq 0}\sum_{m_1|cm}\frac{A_F(m_1,m_2)}{m_1|m_2|} S\pra{am,m_2;\tfrac{cm}{m_1}}\Omega\pra{\tfrac{m_1^2|m_2|}{c^3|m|}},\label{intro_voronoi1}
\end{align}
where $S(\;,\;;\;)$ denotes classical Kloosterman sum.
Secondly, we develop a 
Voronoi formula of Miller-Schmid type (Theorem \ref{voronoi2}) for automorphic forms on $\GL{3}$ of level $N$ for the congruence subgroup $\Gamma_0(N)$ when the conductor of the additive character is coprime with the level $N$
(i.e., $(c,N)=1$, analogous to the second scenario). 
For $(m,N)=(c,N)=1$ we have the formula 
	\begin{align}\nonumber
	\sum_{n=1}^\infty & A_F(m,n) \exp\pra{2\pi i\frac{\bar a n}{c}} \omega(n)
	\\&=c\psi(cm) \epsilon(F) \sum_\pm\sum_{m_2=1}^\infty\sum_{m_1|cm}
	\frac{ A_{\tilde{F}}(m_1,m_2)}{m_1m_2} S\pra{\pm\bar N  am,m_2;\tfrac{cm}{m_1}}\Omega_\pm \pra{\tfrac{m_1^2m_2}{c^3mN}},\label{intro_voronoi2}
	\end{align}
where $\tilde{F}$ is the contragredient form of $F$ and $\bar N N\equiv 1 \;\mod c$. 

\nothing{
\blue{\begin{remark}
I am able to switch the position of $m$ with that of $n$ in \eqref{intro_voronoi2} but I am not able to do that in \eqref{intro_voronoi1}. So now \eqref{intro_voronoi1} looks different than the usual Voronoi formula of Miller-Schmid.
\end{remark}
\begin{remark}
The ultimate goal would be to combine \eqref{intro_voronoi1} and \eqref{intro_voronoi2} into one formula like \cite{kmv}, i.e.,
$$(c, \tfrac{N}{(c,N)})=1.$$
Ideally the goal could be reached by using Atkin-Li operator (which does not exist on $\GL 3$ yet). But now, I don't even have Fricke involution on $\GL 3$. 
\end{remark}
}
}

The author believes that the analogous formulas of Miller-Schmid type should also apply to automorphic forms on $\ggl n R$ for the congruence subgroup $\Gamma_0(N)$. 

The proof of \eqref{intro_voronoi1} is similar to and different from that of \cite{goldfeldli1} by Goldfeld and Xiaoqing Li. In \cite[Section 3]{goldfeldli1}, 
for a Maass cusp form $F$ for $\ssl 3 Z$,  its Voronoi formula involves Fourier coefficients of both $F$ and its contragredient $\tilde F$. 
This would prevent its proof from being generalized to the case of level $N$.
In this paper we realize that the contragredient $\tilde F$ is unnecessary in the formulation and the proof of the Voronoi formula. 
Moreover, in \cite{goldfeldli1}, the proof is based upon that $\tilde F$ is invariant under $\pra{\begin{smallmatrix}
&&1\\1&&\\&1&
\end{smallmatrix}}$ or another Weyl group element,
whereas in Theorem \ref{main} we 
use the invariance of $F$ under
$\pra{\begin{smallmatrix}
1&&\\
&a&b\\
&c&d
\end{smallmatrix}}$. The other parts of our proof are heavily indebted to \cite{goldfeldli1}, and other sources such as Goldfeld and Thillainadesan's proof of the converse theorem on $\GL 3$  in \cite{goldfeld}.


The classical theory of Maass cusp forms for $\ssl 3 Z$ has been developed in \cite{bump, goldfeld}. 
The classical theory of Maass forms on $\GL 3$ for $\Gamma_0(N)$ is still in its infancy. 
The dissertation  \cite{balakci} of Balakci studies the Eisenstein series on $\GL 3$ for $\Gamma_0(N)$.
So does \cite[Chapter 12]{goldfeldhundley} define the classical theory of automorphic forms for $\ggl n R$. 
We will not introduce a precise definition of cuspidality for Maass forms for $\Gamma_0(N)$ because that is unnecessary for the Voronoi formula in Theorem \ref{main}.
For the Voronoi formula of Theorem \ref{main} we only need Maass forms being cuspidal at one cusp, instead of being cuspidal at every cusp. 
In Section \ref{section_background}, we will  define Maass forms on $\GL 3$ for $\Gamma_0(N)$ and its Hecke operators.

Later, the proof of \eqref{intro_voronoi2} uses the method of double Dirichlet series and Gauss sums of the author and Eren Mehmet K{\i}ral in \cite{kiralzhou}.
Theorem \ref{voronoi2} carries on the same formalism from \cite{kiralzhou} and it is based upon a different set of axioms than those used in Theorem \ref{main}. 
In \cite{kiralzhou}, for an automorphic form $f$ for $\ssl n Z$, the Voronoi formulas are proved to be equivalent to the family of functional equation of $L(s,f\times \chi)$ when $\chi$ varies over all the Dirichlet characters. 
In Theorem \ref{voronoi2} we exploit this equivalence further by allowing $f$ to have ramification at finite places.
 Theoretically \eqref{intro_voronoi2} could be recovered from \cite{ichinotemplier} if everything is evaluated explicitly and matched with the classical language.

From the point of view of the theory of automorphic representations, the contrast between the first scenario and the second scenario can be explained by $L$-factor and $\epsilon$-factor.
Let $\pi_p$ be a representation of $\GL n(\mathbb{Q}_p)$ and $\chi_p$  a representation of $\GL 1 (\mathbb{Q}_p)$ which will vary.
If $\pi_p$ is unramified, then $L(s,\pi_p\times \chi_p)$ and $\epsilon(s,\pi_p\times \chi_p)$ 
are very predicable when $\chi_p$ varies over all representations of $\GL 1 (\mathbb{Q}_p)$.
If $\pi_p$ is ramified, then 
$L(s,\pi_p\times \chi_p)$ and $\epsilon(s,\pi_p\times \chi_p)$ 
are very predicable when $\chi_p$ varies over all unramified representations of $\GL 1 (\mathbb{Q}_p)$.
That is how the second scenario works in Theorem \ref{voronoi2} and \cite{kiralzhou}.
Naturally we would like to explore this further. If $\pi_p$ is ramified, unfortunately 
$L(s,\pi_p\times \chi_p)$ and $\epsilon(s,\pi_p\times \chi_p)$ 
are not fully predicable when $\chi_p$ varies over ramified representations of $\GL 1 (\mathbb{Q}_p)$.
In the classical theory on $\GL 2$, this problem is largely resovled by the Atkin-Lehner-Li theory, especially those in \cite{atkinli}.
On a side note,  the stability of the local gamma and epsilon factors at a non-archimedean place should predict that 
when $\chi_p$ is ramified enough relative to $\pi_p$, we have
$$L(s,\pi_p\times \chi_p)\equiv 1$$
and $$\gamma(s,\pi_p\times \chi_p):= \frac{\epsilon(s,\pi_p\times\chi_p)L(1-s,\tilde{\pi}_p\times\bar\chi_p)}{L(s,\pi_p\times\chi_p)}$$ 
ignore most information about $\pi_p$,
as in \cite[\S 4]{atkinli}, \cite{jacquetshalika} and Section \ref{sectionapplication}, as an application of Theorem \ref{main}. 
In Theorem \ref{theorem_L_function}, we will prove that when a Dirichlet character $\chi$ is ramified enough at $p|N$ then
we have the functional equation
$$L\levelN(s,F\times \chi)L_\infty(s,F\times \chi)=c^{-3s}\tau(\psi\chi)\tau(\chi)^2\cdot L\levelN(1-s,\tilde{F}\times\bar\chi)L_\infty(1-s,\tilde{F}\times\bar\chi),$$
where $L\levelN(s,F\times \chi)$ only carries the unramified information of $F$.

Lastly, the absence of the Atkin-Li operator on $\GL 3$ makes it unlikely to combine \eqref{intro_voronoi1} and \eqref{intro_voronoi2} into one formula, as \cite{kmv} does on $\GL 2$. 
It is very valuable to develop a newform theory on $\GL 3$ and $\GL n$.

\subsection*{Acknowledgment}
The author would like to thank Jack Buttcane, James Cogdell, Eren Mehmet K{\i}ral, Yuk-Kam Lau, Stephen David Miller for helpful discussion.
\nothing{
\section{The GL(2) cases}
Define $e(x):=e^{2\pi i x}$.
Let $f$ be a cuspidal modular form of weight $k$ for $\mathbf{\Gamma}_0(N)\subseteq \ssl 2 Z$. 
It has Fourier expansion
$$f(z)=\sum_{n=1}^\infty \lambda_f(n) n^{\frac{k-1}{2}}e^{2\pi i z}.$$
Let $\pra{\begin{smallmatrix}
a&b\\c&d
\end{smallmatrix}}\in \mathbf{\Gamma}_0(N) $ with $N|c$. 
We have $$f\pra{\frac{az+b}{cz+d}}=(cz+d)^kf(z).$$
Take $z=-\frac{d}{c}+iy$. 
We have $$f\pra{\frac{a}{c}+\frac{i}{c^2y}}=(icy)^kf\pra{-\frac{d}{c}+iy}$$
Applying Mellin transform to both sides, we get
$$\pra{\frac{c}{2\pi}}^s\Gamma\pra{s+\frac{k-1}{2}}L(s,f,\frac{a}{c})=i^{-k}\pra{\frac{c}{2\pi}}^{1-s}
\Gamma\pra{1-s+\frac{k-1}{2}}
L(1-s,f,-\frac{\bar a}{c}),$$
where $L(s,f,\frac{a}{c})=\sum_n \lambda_f(n)e(\frac{an}{c})n^{-s}$.

Let $f$ be a cuspidal modular form of weight $k$ for $\mathbf{\Gamma}_0(N)\subseteq \ssl 2 Z$. 
Define $g(z):=(\sqrt{N}z)^{-k}f(-\frac 1 {Nz})$, which is a cuspidal modular form of weight $k$ for $\mathbf{\Gamma}_0(N)$.
Let $\pra{\begin{smallmatrix}
a&b\\c&d
\end{smallmatrix}}\in \mathbf{\Gamma}_0(N) $ with $N|c$. 
We have 
$$f(z)=\pra{\sqrt{N}\pra{-c/N+dz}}^{-k}g\pra{\frac{-a/N+bz}{-c/N+dz}}.$$
Take $z=\frac{c}{Nd}+iy$. 
We have $$f\pra{\frac{c}{Nd}+iy}=(i\sqrt N dy)^{-k}g\pra{\frac{b}{d}+\frac{i}{Nd^2y}}$$
Applying Mellin transform to both sides, we get
$$\pra{\frac{d\sqrt{N}}{2\pi}}^s\Gamma\pra{s+\frac{k-1}{2}}L(s,f,\frac{c}{dN})=i^{-k}\pra{\frac{d\sqrt{N}}{2\pi}}^{1-s}
\Gamma\pra{1-s+\frac{k-1}{2}}
L(1-s,g,\frac{b}{d}),$$
i.e.,
$$\pra{\frac{d\sqrt{N}}{2\pi}}^s\Gamma\pra{s+\frac{k-1}{2}}L(s,f,-\frac{\bar b\bar N}{d})=i^{-k}\pra{\frac{d\sqrt{N}}{2\pi}}^{1-s}
\Gamma\pra{1-s+\frac{k-1}{2}}
L(1-s,g,\frac{b}{d}),$$
where $\bar x$ is the multiplicative inverse of $x$ mod $d$.

}

\section{The first scenario}
\subsection{Background on automorphic forms}\label{section_background}
The classical theory of modular forms for the congruence subgroup $\Gamma_0(N)\subset \ssl 2 Z$ has been well developed and recorded in excellent expositions, such as \cite{miyake}. 

The theory of Maass cusp forms for  $\ssl 3 Z$ and $\ssl n Z$ has been well recorded by \cite{bump, goldfeld} and these Maass forms are corresponding to unramified automorphic representations of $\ggl n {A_Q}$. 
It is desirable to develop the theory of automorphic forms on $\ggl n  R$ for the congruence subgroup $\Gamma_0(N)\subset \ssl n Z$, and they will correspond to automorphic representations of $\ggl n {A_Q}$ with ramification at finite places. 
In the following we will discuss how Maass forms on $\ggl 3  R$ for $\Gamma_0(N)$ should be developed. Many results in \cite{bump, goldfeld} can be carried on to $\Gamma_0(N)$ with no or little modification. I assume that the reader is familiar with \cite[Section 5, 6]{goldfeld}.

Let $\mathfrak{h}=\ggl 3 R/(\mathrm O _3(\mathbb R)\cdot \mathbb R ^\times)$
denote the generalized upper half plane. Each $z\in \mathfrak h$ has the form
$z=xy$ by the Iwasawa decomposition, where 
$$x=\begin{pmatrix}
1&x_2&x_3\\
&1&x_1\\
&&1
\end{pmatrix}\quad \text{ and } \quad
y=\begin{pmatrix}
y_1y_2&&\\
&y_1&\\
&&1
\end{pmatrix}
$$
and $x_i, y_i \in \mathbb R$ and $y_i >0$.	
Let $\mathcal{ D}$ denote the  the center of the universal
enveloping algebra of $\mathrm{\mathfrak{gl}}_3(\mathbb{R})$. 
Define the congruence subgroup $\Gamma_0(N)$ of $\ssl 3 Z$ by 
$$\Gamma_0(N) :=\left\{g\in \ssl 3 Z:g\equiv \begin{pmatrix}
*&*&*\\ *&*&*\\ 0&0&*\end{pmatrix}  \mod N \right\}.$$

The Hilbert space $\mathcal{L}^2(\Gamma_0(N)\setminus\mathfrak{h},\psi)$
denotes the space of functions $F$
with 
$$\int_{\Gamma_0(N)\setminus \mathfrak h}  |F(z)|^2\dd^* z<\infty$$
where $\dd^* z=\dd x_1\dd x_2\dd x_3 \tfrac{\dd y_1\dd y_2}{(y_1y_2)^3}$ and 
$$F\pra{gz}=\psi(g_{33})F(z), \text{ for all }g=\begin{pmatrix}
g_{11 }&g_{12 }&g_{13 }\\
g_{21 }&g_{22 }&g_{23 }\\
g_{31 }&g_{32 }&g_{33 }
\end{pmatrix} \in \Gamma_0(N).$$
Define the Petersson inner product on $\mathcal{L}^2(\Gamma_0(N)\setminus\mathfrak{h},\psi)$
$$\left<F,G\right>:=\int_{\Gamma_0(N)\setminus \mathfrak h}  F(z)\overline{G(z)}\dd^* z.$$

A Maass form for $\Gamma_0(N)$ with a Dirichlet character $\psi: \mathbb{Z}/N\mathbb{Z}\to \mathbb {C}^\times$
is a smooth function $F\in \mathcal{L}^2(\Gamma_0(N)\setminus\mathfrak{h},\psi)$ which satisfies 
\begin{itemize}
\item $F\pra{gz}=\psi(g_{33})F(z), \text{ for all }z\in \mathfrak h, \; g=\begin{pmatrix}
g_{11 }&g_{12 }&g_{13 }\\
g_{21 }&g_{22 }&g_{23 }\\
g_{31 }&g_{32 }&g_{33 }
\end{pmatrix} \in \Gamma_0(N)$;
\item 
for all $D\in \mathcal D$, $F$ is an eigenfunction of $D$;
\item $\int_{0}^1\int_{0}^1 F\pra{\begin{pmatrix}
1&&u_3\\&1&u_1\\&&1
\end{pmatrix}
z}\dd u_1\dd u_3=0$
 and 
$\int_{0}^1\int_{0}^1 F\pra{\begin{pmatrix}
1&u_2&u_3\\&1&\\&&1
\end{pmatrix}
z}\dd u_2\dd u_3=0$ for all $z\in \mathfrak{h}$.
\end{itemize}

\begin{remark}
We do not call $F$ a Maass \textit{cusp} form because the last condition only requires cuspidality at one cusp.
\end{remark}

Define $e(x):=\exp(2\pi ix)$.
Because $F$ is invariant under $\pra{
\begin{smallmatrix} * & * & *\\
* & * & *\\
0 & 0 & *\\
\end{smallmatrix}
  }$, it has Fourier-Whittaker expansion by \cite[Theorem 0.3]{balakci} and \cite[Theorem 5.3.2]{goldfeld}, 
\begin{align}\nonumber
F(z)&
\nothing{=\sum_{\pra{\begin{smallmatrix}A&B\\C&D\end{smallmatrix}}  \in U_2(\mathbb{Z})\setminus \ssl 2 Z} \sum_{m_1=1}^\infty \sum_{m_2\neq 0} \frac{A_F(m_1,m_2)}{m_1|m_2|}
\WJ\pra{
\begin{pmatrix}
m_1|m_2| & &\\
&m_1&\\
&&1
\end{pmatrix}
\begin{pmatrix}
A&B &\\
C&D&\\
&&1
\end{pmatrix}
z}\\}
=\sum_{
\pra{\begin{smallmatrix}A&B\\C&D\end{smallmatrix}}
\in U_2(\mathbb{Z})\setminus \ssl 2 Z} \sum_{m_1=1}^\infty \sum_{m_2\neq 0} \frac{A_F(m_1,m_2)}{m_1|m_2|}e\pra{m_1\pra{Cx_3+Dx_1}+m_2\Re \frac{A z_2+B}{C z_2 +D}} 
\\
&\spaces \quad  \;\quad \;\quad \WJ\pra{\begin{pmatrix}
m_1|m_2| & &\\
&m_1&\\
&&1
\end{pmatrix}
\begin{pmatrix}
\frac{y_1y_2}{|Cz_2+D|} & &\\
&y_1|Cz_2+D|&\\
&&1
\end{pmatrix}
}.\label{fourierexpansion}
\end{align}

The Jacquet's Whittaker  function $\WJ$ is defined in \cite[Section 5.5]{goldfeld}. We only need the Jacquet's Whittaker function on $\GL 3$, which can also be found in \cite[(2.10)]{blomer}.  We follow the notation \textit{ibid.} for $\mathcal{W}^\pm_{\nu_1,\nu_2}$
 and we have
$$\mathcal{W}^\pm_{\nu_1,\nu_2}\pra{\begin{pmatrix}
1&x_2&x_3\\
&1&x_1\\
&&1
\end{pmatrix}
\begin{pmatrix}
y_1y_2&&\\
&y_1&\\
&&1
\end{pmatrix}
}
=e(x_1\pm x_2)\WJ\pra{\begin{pmatrix}
y_1y_2&&\\
&y_1&\\
&&1
\end{pmatrix}
}.
$$
Here we suppress the spectral parameter $(\nu_1, \nu_2)$ from $\WJ$. 
The spectral parameter $(\nu_1, \nu_2)$ is determined by the eigenvalues of $F$ under the differential operators $D\in \mathcal{D}$. In the rest of this paper, we will only use $\WJ$ but not $\mathcal{W}_{\nu_1,\nu_2}^\pm$. 
We define for convenience
\begin{align*}
&\alpha:=-\nu_1-2\nu_2+1\\
&\beta:=-\nu_1+\nu_2\\
&\gamma:=2\nu_1+\nu_2-1
\end{align*}
and 
$$\mathsf C:=\pi^{\frac 1 2- 3\nu_1-3\nu_2}\Gamma\pra{\frac{3\nu_1}{2}}\Gamma\pra{\frac{3\nu_2}{2}}\Gamma\pra{\frac{3\nu_1+3\nu_2-1}{2}}.$$

The Fourier-Whittaker coefficient $A_F(m_1,m_2)$ can be obtained by 
\begin{align}\nonumber
\int_0^1&\int_0^1\int_0^1
F\pra{
\begin{pmatrix}
1&u_2&u_3\\
&1&u_1\\
&&1
\end{pmatrix}
z}e(-m_1u_1-m_2u_2)\dd u_1\dd  u_2\dd u_3
\\&=\frac{A_F(m_1,m_2)}{|m_1m_2|}e(m_1x_1+m_2x_2)\WJ\pra{\begin{pmatrix}
|m_1m_2| & &\\
&|m_1|&\\
&&1
\end{pmatrix}
\begin{pmatrix}
y_1y_2&&\\
&y_1&\\
&&1
\end{pmatrix}
}\label{fourierwhittaker}
\end{align}
for non-zero integers $m_1$ and $m_2$.

Because of 
$$F(z)=F\pra{\begin{pmatrix}
-1&&\\&-1&\\&&1
\end{pmatrix}
z}=\psi(-1)F\pra{\begin{pmatrix}
-1&&\\&1&\\&&-1
\end{pmatrix}
z}=\psi(-1)F\pra{\begin{pmatrix}
1&&\\&-1&\\&&-1
\end{pmatrix}
z}$$
and 
$$
\begin{pmatrix}
\delta_1&&\\&\delta_2&\\&&\delta_3
\end{pmatrix} 
\begin{pmatrix}
1&x_2&x_3
\\
&1&x_1\\
&&1
\end{pmatrix}
\begin{pmatrix}
\delta_1&&\\&\delta_2&\\&&\delta_3
\end{pmatrix} ^{-1}
=\begin{pmatrix}
1&\delta_1\delta_2^{-1}x_2&\delta_1\delta_3^{-1}x_3
\\
&1&\delta_2\delta_3^{-1} x_1\\
&&1
\end{pmatrix}
$$
we have the following from \eqref{fourierwhittaker}. 
This is analogous to \cite[Proposition 6.3.5]{goldfeld}

\begin{proposition}\label{even_odd}
Let $F$ be a Maass form for $\Gamma_0(N)$ with a Dirichlet character $\psi$.
We have $$A_F(\pm m_1,(-1)^k m_2)=\psi(-1)^k A_F(m_1,m_2)$$
for $k=0,1$ and positive integers $m_1$ and $m_2$.
\end{proposition}

For 
$
z_1
=x_1+iy_1
$
and for $\sigma=\begin{pmatrix}
a&b\\
c&d
\end{pmatrix} \in \ssl 2 Z$
we will use the notation for linear fractional transformation
$$\sigma z_1 = \frac{az_1+b}{cz_1+d}.$$
We will use the two basic facts 
$$\Re \sigma z_1 = \frac{a}{c}-\Re\frac{1}{c(cz_1+d)}$$
and
$$\Im \sigma z_1 = \frac{\Im z_1}{|cz_1+d|^2}. $$

\subsection{Hecke operators}\label{subsection_hecke_operator}
Let $F$ be a Maass form for $\Gamma_0(N)$ and a Dirichlet character $\psi$, as defined in Section \ref{section_background}. 
For a positive integer $n$, define the Hecke operator $T_n$ by
\begin{equation}
T_n F(z):=\frac{1}{n}\sum_{abc=n}\sum_{0\leq b_1<b}\sum_{0\leq c_1,c_2<c}
\psi(ab)F\pra{
\begin{pmatrix}
a&b_1&c_1\\
&b&c_2\\
&&c
\end{pmatrix}
z}.
\end{equation}
Let $T^*_n$ be the adjoint operator of $T_n$ which satisfies
$$\left<T_n F, G\right>=\left<F,T^*_n G\right >$$
for any $F, G\in \mathcal{L}^2(\Gamma_0(N)\setminus\mathfrak{h},\psi)$. 
Explicitly the adjoint operator $T_n^*$ is 
\begin{equation}
T^*_n F(z):=\frac{1}{n}\sum_{abc=n}\sum_{0\leq b_1<b}\sum_{0\leq c_1,c_2<c}
\overline{\psi(ab)}F\pra{
\begin{pmatrix}
a&b_1&c_1\\
&b&c_2\\
&&c
\end{pmatrix}^{-1}
z}.
\end{equation}

Let us assume that $F$ is an eigenfunction under $T_n$ and $T^*_n$ for all $n $ with $(n,N)=1$.
Let us assume  $A_F(1,1)=1$ for normalization. Admittedly there could be Maass forms for $\Gamma_0(N)$ with $A_F(1,1)=0$. But so far the newform theory or the Atkin-Lehner-Li theory on $\GL 3$ has not been built. 
By the standard method of automorphic forms (\cite[Theorem 6.4.11]{goldfeld}), we have 
$$T_nF=A_F(n,1)F.$$
Moreover we have
$$T^*_nF=\overline{\psi(n)}A_F(1,n)F    $$
and 
$$A_F(n,1)=\psi(n)\overline{A_F(1,n)}$$
with $(n,N)=1$ after comparing the eigenvalues of $T_n$ and $T_n^*$.
By the same method we have the Hecke relations
\begin{equation}
A_F(n,1)A_F(n_2, n_1)=\sum_{abc=n, a|n_1, b|n_2}\psi (ab)A_F(\tfrac{cn_2}b,\tfrac{ bn_1}a) \label{heckerelationsection1}
\end{equation}
and
\begin{equation}\label{heckerelationsection2}
A_F(1,n)A_F(n_2, n_1)=\sum_{abc=n, b|n_1, c|n_2}\psi (c)A_F(\tfrac{bn_2}c,\tfrac{ an_1}b)
\end{equation}
with $(n,N)=1$ and $n_1,n_2\in\mathbb Z$.

\begin{proposition}
Let us assume that the Maass form $F$ is an eigenfunction under $T_n^*$ for some $n$ with $(n,N)=1$ and assume $A_F(1,1)=1$. 
Then we have 
			$$T^*_nF=\overline{\psi(n)}A_F(1,n)F    $$
and 
			$$A_F(1,n)A_F(n_2, n_1)=\sum_{abc=n, b|n_1, c|n_2}\psi (c)A_F(\tfrac{bn_2}c,\tfrac{ an_1}b).$$
\end{proposition}

\begin{proof}The proof would be very similar to that of \cite[Theorem 6.4.11]{goldfeld}
\end{proof}

Let $\tilde{F}$ be the hypothetical contragredient form of $F$. 
Let $A_{\tilde F}(m,n)$ be the Fourier-Whittaker coefficient of $\tilde{F}$ and we have for positive integers $m,n$ with $(mn,N)=1$
$$A_F(m,n)=\psi(m)\psi(n)A_{\tilde F}(n,m).$$
Unlike the case on $\GL 2$ we so far are not able to define $\tilde{F}$ precisely when $N>1$.

\subsection{The left side}
Let $m\neq 0$ be an integer. 
To prove \eqref{intro_voronoi1} we will construct the integral $H(x_1;y_1,y_2)$, which is similar to  but  different from
\cite[(3.3)]{goldfeldli1}. We define the integral 
$$H(x_1;y_1,y_2):=\psi(d)\int_0^c\int_0^1    F\pra{\begin{pmatrix}
1&x_2&x_3+\frac{dx_2}{c}+x_1x_2\\
&1&x_1\\
&&1
\end{pmatrix}
\begin{pmatrix}
y_1y_2&&\\
&y_1&\\
&&1
\end{pmatrix}} e(-mx_2) \dd x_3 \dd x_2.$$

\begin{lemma}
For $k=0$ or $1$ and for fixed $y_2$,  the function 
$\left.\pra{\frac{\partial }{\partial x_1}}^k H(x_1;y_1,y_2)\right|_{x_1=-\frac{d}{c}}$
has rapid decay as $y_1\to \infty$.	
\end{lemma}

\begin{proof}
By \eqref{fourierexpansion}, $F$ in $H(x_1;y_1,y_2)$ has the Fourier-Whittaker expansion
\begin{align*}
F&\pra{\begin{pmatrix}
1&x_2&x_3+\frac{dx_2}{c}+x_1x_2\\
&1&x_1\\
&&1
\end{pmatrix}
\begin{pmatrix}
y_1y_2&&\\
&y_1&\\
&&1
\end{pmatrix}} \spaces \spaces
\\
=&\sum_{\pra{\begin{smallmatrix}
A&B\\C&D
\end{smallmatrix}}\in U_2(\mathbb{Z})\setminus \ssl 2 Z} \sum_{m_1=1}^\infty \sum_{m_2\neq 0} \frac{A_F(m_1,m_2)}{m_1|m_2|}e\pra{m_1\pra{C\pra{x_3+\frac{dx_2}{c}+x_1x_2}+Dx_1}+m_2\Re \frac{A z_2+B}{C z_2 +D}} 
\\
&\spaces \spaces
\WJ\pra{\begin{pmatrix}
m_1|m_2| & &\\
&m_1&\\
&&1
\end{pmatrix}
\begin{pmatrix}
\frac{y_1y_2}{|Cz_2+D|} & &\\
&y_1|Cz_2+D|&\\
&&1
\end{pmatrix}}.
\end{align*}
In $H(x_1;y_1,y_2)$, the integration in $\dd x_3$ forces $m_1C=0$. This implies $C=0$ and then $A=D=\pm 1$. Integration in $\dd x_2$ forces $m_2=m$. Thus we have 
$$H(x_1;y_1,y_2)=c\psi(d)\sum_{m_1=1}^\infty \frac{A_F(m_1,m)}{m_1|m|} e(\pm m_1 x_1)\WJ\pra{\begin{pmatrix}
m_1|m| & &\\
&m_1&\\
&&1
\end{pmatrix}
\begin{pmatrix}
y_1y_2 & &\\
&y_1&\\
&&1
\end{pmatrix}} . $$
Therefore $\left.\pra{\frac{\partial }{\partial x_1}}^k H(x_1;y_1,y_2)\right|_{x_1=-\frac{d}{c}}$ has rapid decay as $y_1\to \infty$ because of the decay property of the Jacquet's Whittaker function.
\end{proof}

For $k=0$ or $1$, we define 
$$
H_k(s_1,s_2):=
\int_0^\infty\int_0^\infty \left.\pra{\frac{\partial }{\partial x_1}}^k H(x_1;y_1,y_2)\right|_{x_1=-\frac{d}{c}} y_1^{s_1-1}y_2^{s_2-1}\frac{\dd y_1}{y_1} \frac{\dd y_2}{y_2}.$$
By a change of variable we get for $\Re s_1\gg 1$ 
\begin{align}\label{dirichlet1}
H_k(s_1,s_2)&=c\psi(d)\sum_{m_1=1}^\infty \frac{A_F(m_1,m)}{m_1^{s_1}|m|^{s_2}} (2\pi im_1)^k \pra{ e\pra{-\frac {m_1d}{c} }+(-1)^k e\pra{\frac {m_1d}{c} }} \\
 &\spaces  \int_0^\infty\int_0^\infty\WJ\pra{
\begin{pmatrix}
y_1y_2 & &\\
&y_1&\\
&&1
\end{pmatrix}
}y_1^{s_1-1}y_2^{s_2-1}\frac{\dd y_1}{y_1} \frac{\dd y_2}{y_2}.\nonumber
\end{align}

\begin{lemma}
For $\Re s_1\gg 1$, $H_k(s_1,s_2)$ is absolutely convergent for all $s_2\in \mathbb{C}$. 
\end{lemma}
\subsection{The right side}
Let $c$ be a multiple of $N$. Let $\sigma=\begin{pmatrix}
a&b\\c&d
\end{pmatrix}$ be a matrix in $\ssl 2 Z$ and thus $\begin{pmatrix}
1&\\&\sigma
\end{pmatrix}\in \Gamma_0(N)$.    
We have the Iwasawa decomposition
\begin{align*}
&\begin{pmatrix}
1&&\\
&a&b\\
&c&d
\end{pmatrix}
\begin{pmatrix}
1&x_2&x_3+\frac{dx_2}{c}+x_1x_2\\
&1&x_1\\
&&1
\end{pmatrix}
\begin{pmatrix}
y_1y_2&&\\
&y_1&\\
&&1
\end{pmatrix}\\
=&\begin{pmatrix}
1&-cx_3 & ax_3+\frac{x_2}{c }-cx_3\Re\sigma z_1\\
&1 & \Re \sigma z_1\\
&&1
\end{pmatrix}
\begin{pmatrix}
y_2|cz_1+d|\Im \sigma z_1&&\\
&\Im \sigma z_1&\\
&&1
\end{pmatrix} 
\mod \pra{\oo 3 R \cdot \mathbb{R }^\times}. 
\end{align*}
We will re-calculate $H(x_1;y_1,y_2)$ using the invariance of $F$ under 
$\begin{pmatrix}
1&\\&\sigma
\end{pmatrix}\in \Gamma_0(N)$. This is crucially different from \cite[(3.12)]{goldfeld}. 

\begin{lemma}
For $k=0$ or $1$ and for fixed $y_2$,  the function 
$\left.\pra{\frac{\partial }{\partial x_1}}^k H(x_1;y_1,y_2)\right|_{x_1=-\frac{d}{c}}$
has rapid decay as $y_1\to 0$.	
\end{lemma}
\begin{proof}
Using the aforementioned Iwasawa decomposition, we have the Fourier-Whittaker expansion 
\begin{align*}
&F\pra{\begin{pmatrix}
1&&\\
&a&b\\
&c&d
\end{pmatrix}\begin{pmatrix}
1&x_2&x_3+\frac{dx_2}{c}+x_1x_2\\
&1&x_1\\
&&1
\end{pmatrix}
\begin{pmatrix}
y_1y_2&&\\
&y_1&\\
&&1
\end{pmatrix}}\\
&=\sum_{C=-\infty}^{\infty} \sum_{\underset{\scriptstyle (D,C)=1}{ D=-\infty}}^{\infty} \sum_{m_1=1}^\infty \sum_{m_2\neq 0} \frac{A_F(m_1,m_2)}{m_1|m_2|}e\pra{m_1\pra{C\pra{ax_3+\frac{x_2}{c }-cx_3\Re\sigma z_1} 
+D \Re \sigma z_1} +m_2\Re \frac{A z'_2+B}{C z'_2 +D}} 
\\&\spaces
\spaces
\WJ\pra{\begin{pmatrix}
m_1|m_2| & &\\
&m_1&\\
&&1
\end{pmatrix}
\begin{pmatrix}
\frac{y_2|cz_1+d|\Im \sigma z_1}{|Cz'_2+D|} & &\\
&\Im \sigma z_1|Cz'_2+D|&\\
&&1
\end{pmatrix}
},
\end{align*}
where we define
$z'_2:=-cx_3+y_2|cz_1+d|i$.
Here $A$ and $B$ are a pair of integers satisfying $AD-BC=1$. 
We compute the part in $e(\cdots )$, 
\begin{align*}
&m_1\pra{C\pra{ax_3+\frac{x_2}{c }-cx_3\Re\sigma z_1} 
+D \Re \sigma z_1} +m_2\Re \frac{A z'_2+B}{C z'_2 +D}\\
=&\frac{m_1Cx_2}{c}+
\frac{m_1Da}{c}+\frac{m_2A}{ C}+m_1(Ccx_3-D)\Re\frac{1}{c(cz_1+d)}-m_2\Re\frac{1}{C(Cz_2'+D)}.
\end{align*}
Because of $F\pra{\pra{\begin{smallmatrix}
1&&\\
&a&b\\
&c&d
\end{smallmatrix}}z}=\psi(d)F(z)$, we have another way to calculate $H(x_1;y_1,y_2)$ and it is 
$$H(x_1;y_1,y_2)=\int_0^1  \int_0^c  F\pra{\begin{pmatrix}
1&&\\
&a&b\\
&c&d
\end{pmatrix}\begin{pmatrix}
1&x_2&x_3+\frac{dx_2}{c}+x_1x_2\\
&1&x_1\\
&&1
\end{pmatrix}
\begin{pmatrix}
y_1y_2&&\\
&y_1&\\
&&1
\end{pmatrix}} e(-mx_2)\dd x_2\dd x_3,$$
after switching $\dd x_3 $ and $\dd x_2$. 
In $H(x_1;y_1,y_2)$, the integration in $\dd x_2$ forces $m_1C=cm$ and we have 
\begin{align*}
&\;H(x_1;y_1,y_2)
\\
&=c\int_0^1 
\sum_{\underset{\scriptstyle m_1C=cm}{m_1,C\in \mathbb Z,\; m_1>0}}
\sum_{\underset{\scriptstyle (D,C)=1}{ D=-\infty}}^{\infty}  \sum_{m_2\neq 0} \frac{A_F(m_1,m_2)}{m_1|m_2|}
\\&\spaces e\pra{\frac{Dam	}{C}+\frac{m_2A}{ C}+m_1(Ccx_3-D)\Re\frac{1}{c(cz_1+d)}-m_2\Re\frac{1}{C(Cz_2'+D)}}
\\
&\spaces\spaces
\WJ\pra{\begin{pmatrix}
m_1|m_2| & &\\
&m_1&\\
&&1
\end{pmatrix}
\begin{pmatrix}
\frac{y_2|cz_1+d|\Im \sigma z_1}{|Cz'_2+D|} & &\\
&\Im \sigma z_1|Cz'_2+D|&\\
&&1
\end{pmatrix}
}\dd x_3.
\end{align*}
The $D$-summation $\sum\limits_{\underset{\scriptstyle (D,C)=1}{ D=-\infty}}^{\infty} $ can be rewritten as a summation of 
$D+lC$ with $D$ ranging in $\sum\limits_{\underset{\scriptstyle (D,C)=1}{ D=1}}^{C} $ and $l$ ranging over all integers.
We now obtain
\begin{align*}
&c\int_0^1 
\sum_{\underset{\scriptstyle m_1C=cm}{m_1,C\in \mathbb Z,\; m_1>0}}
\sum_{\underset{\scriptstyle (D,C)=1}{ D=1}}^{C} \sum_{l=-\infty}^\infty  \sum_{m_2\neq 0} \frac{A_F(m_1,m_2)}{m_1|m_2|}
\\&\spaces e\pra{\frac{Dam	}{C}+\frac{m_2A}{ C}+m_1(Ccx_3-D-lC)\Re\frac{1}{c(cz_1+d)}-m_2\Re\frac{1}{C(Cz_2'+D+lC)}}
\\
&\spaces\spaces
\WJ\pra{\begin{pmatrix}
m_1|m_2| & &\\
&m_1&\\
&&1
\end{pmatrix}
\begin{pmatrix}
\frac{y_2|cz_1+d|\Im \sigma z_1}{|Cz'_2+D+lC|} & &\\
&\Im \sigma z_1|Cz'_2+D+lC|&\\
&&1
\end{pmatrix}
}\dd x_3.
\end{align*}
Changing variables $x_3\to x_3+\frac l c+\frac D {cC}$ we obtain
\begin{align*}
&c
\sum_{\underset{\scriptstyle m_1C=cm}{m_1,C\in \mathbb Z,\; m_1>0}}
\sum_{\underset{\scriptstyle (D,C)=1}{ D=1}}^{C} \sum_{l=-\infty}^\infty \int_{-\frac l c-\frac D {cC}}^{1-\frac l c-\frac D {cC}}  \sum_{m_2\neq 0} \frac{A_F(m_1,m_2)}{m_1|m_2|}
\\&\spaces e\pra{\frac{Dam	}{C}+\frac{m_2A}{ C}+m_1Ccx_3\Re\frac{1}{c(cz_1+d)}-m_2\Re\frac{1}{C^2z_2'}}
\\
&\spaces\spaces
\WJ\pra{\begin{pmatrix}
m_1|m_2| & &\\
&m_1&\\
&&1
\end{pmatrix}
\begin{pmatrix}
\frac{y_2|cz_1+d|\Im \sigma z_1}{|Cz'_2|} & &\\
&\Im \sigma z_1|Cz'_2|&\\
&&1
\end{pmatrix}
}\dd x_3
\end{align*}
\begin{align*}
=&\;c^2
\sum_{\underset{\scriptstyle m_1C=cm}{m_1,C\in \mathbb Z,\; m_1>0}}
\sum_{\underset{\scriptstyle (D,C)=1}{ D=1}}^{C} \int_{-\infty}^\infty \sum_{m_2\neq 0} \frac{A_F(m_1,m_2)}{m_1|m_2|}
\\&\spaces e\pra{\frac{Dam	}{C}+\frac{m_2A}{ C}+m_1Ccx_3\Re\frac{1}{c(cz_1+d)}-m_2\Re\frac{1}{C^2z_2'}}
\\
&\spaces\spaces
\WJ\pra{\begin{pmatrix}
m_1|m_2| & &\\
&m_1&\\
&&1
\end{pmatrix}
\begin{pmatrix}
\frac{y_2|cz_1+d|\Im \sigma z_1}{|Cz'_2|} & &\\
&\Im \sigma z_1|Cz'_2|&\\
&&1
\end{pmatrix}
}\dd x_3.
\end{align*}
We realize $$\sum_{\underset{\scriptstyle (D,C)=1}{ D=1}}^{C}
e\pra{\frac{Dam	}{C}+\frac{m_2A}{ C}}=S(am,m_2;C), $$
which is the classical Kloosterman sum. We obtain

\begin{align*}
&\;c^2
\sum_{\underset{\scriptstyle m_1C=cm}{m_1,C\in \mathbb Z,\; m_1>0}}
\sum_{m_2\neq 0} \frac{A_F(m_1,m_2)}{m_1|m_2|} S(am,m_2;C)  \int_{-\infty}^\infty 
e\pra{m_1Ccx_3\Re\frac{1}{c(cz_1+d)}-m_2\Re\frac{1}{C^2z_2'}}
\\
&\spaces\spaces
\WJ\pra{\begin{pmatrix}
m_1|m_2| & &\\
&m_1&\\
&&1
\end{pmatrix}
\begin{pmatrix}
\frac{y_2|cz_1+d|\Im \sigma z_1}{|Cz'_2|} & &\\
&\Im \sigma z_1|Cz'_2|&\\
&&1
\end{pmatrix}
}\dd x_3.
\end{align*}
\nothing{
xxxxxxxxxxx
\begin{align*}
&c^2
\sum_{\underset{\scriptstyle m_1C=cm}{m_1,C\in \mathbb Z,\; m_1>0}}
\sum_{\underset{\scriptstyle (D,C)=1}{ D=1}}^{C}   \sum_{m_2\neq 0}
\int_{-\infty}^\infty 
 \frac{A_F(m_1,m_2)}{m_1|m_2|}\\
&\spaces e\pra{\frac{Dam	}{C}+\frac{m_2A}{ C}+m_1Cx_3\Re\frac{1}{cz_1+d}-m_2\Re\frac{1}{C^2c(-x_3+y_1y_2i)}}
\\
&\spaces
\WJ\pra{\begin{pmatrix}
m_1|m_2| & &\\
&m_1&\\
&&1
\end{pmatrix}
\begin{pmatrix}
\frac{y_2|cz_1+d|\Im \sigma z_1}{|Cc(-x_3+y_1y_2i)|} & &\\
&\Im \sigma z_1|Cc(-x_3+y_1y_2i)|&\\
&&1
\end{pmatrix}
}\dd x_3.
\end{align*}
}
For $k=0$ or $1$, we get after taking partial derivative with respect to $x_1$ at $x_1=-\frac d c$ 
\begin{align*}
& \left.\pra{\frac{\partial }{\partial x_1}}^k H(x_1;y_1,y_2)\right|_{x_1=-\frac{d}{c}}\\
&=c^2        \sum_{\underset{\scriptstyle m_1C=cm}{m_1,C\in \mathbb Z,\; m_1>0}}
 \sum_{m_2\neq 0} \frac{A_F(m_1,m_2)}{m_1|m_2|}S(am,m_2;C)
\int_{-\infty}^{\infty} 
 \pra{2\pi i \frac{m_1Cx_3}{cy_1^2}}^k e\pra{-m_2\Re\frac{1}{C^2c(-x_3+y_1y_2i)}}\\&\spaces\spaces
\WJ\pra{\begin{pmatrix}
m_1|m_2| & &\\
&m_1&\\
&&1
\end{pmatrix}
\begin{pmatrix}
\frac{y_2}{c^2|C||-x_3+iy_1y_2|} & &\\
&\frac{|C||-x_3+y_1y_2i|}{cy_1}&\\
&&1
\end{pmatrix}
}\dd x_3.
\end{align*}
By a change of variable $x_3\to y_1y_2x_3$ we obtain
\begin{align*}
&c^2
\sum_{\underset{\scriptstyle m_1C=cm}{m_1,C\in \mathbb Z,\; m_1>0}}
   \sum_{m_2\neq 0} \frac{A_F(m_1,m_2)}{m_1|m_2|}
S(am,m_2;C)
\int_{-\infty}^\infty
\pra{2\pi i \frac{m_1Cx_3y_2}{cy_1}}^k
 e\pra{ -\frac{m_2}{C^2cy_1y_2}\Re\frac{1}{-x_3+i}}\\
&\spaces\WJ\pra{\begin{pmatrix}
m_1|m_2| & &\\
&m_1&\\
&&1
\end{pmatrix}
\begin{pmatrix}
\frac{1}{c^2|C|y_1|-x_3+i|} & &\\
&\frac{|C||-x_3+i|y_2}{c}&\\
&&1
\end{pmatrix}
}y_1y_2\dd x_3. 
 \end{align*}
 Therefore $\left.\pra{\frac{\partial }{\partial x_1}}^k H(x_1;y_1,y_2)\right|_{x_1=-\frac{d}{c}}$ has rapid decay as $y_1\to 0	$ because of the decay property of the Jacquet's Whittaker function.
\end{proof}

Now we calculate $H_k(s_1,s_2)$ again

\begin{align*}
&\; H_k(s_1,s_2)\\
&=c^2
\int_0^\infty\int_0^\infty
\sum_{\underset{\scriptstyle m_1C=cm}{m_1,C\in \mathbb Z,\; m_1>0}}
   \sum_{m_2\neq 0} \frac{A_F(m_1,m_2)}{m_1|m_2|}
S(am,m_2;C)\\
&\quad\quad\quad\quad
\int_{-\infty}^\infty
\pra{2\pi i \frac{m_1Cx_3y_2}{cy_1}}^k
 e\pra{ -\frac{m_2}{C^2cy_1y_2}\Re\frac{1}{-x_3+i}}\\
&\spaces
\quad\WJ\pra{\begin{pmatrix}
m_1|m_2| & &\\
&m_1&\\
&&1
\end{pmatrix}
\begin{pmatrix}
\frac{1}{c^2|C|y_1|-x_3+i|} & &\\
&\frac{|C||-x_3+i|y_2}{c}&\\
&&1
\end{pmatrix}
}\dd x_3 y_1^{s_1} y_2^{s_2}\frac{\dd y_1}{y_1}\frac{\dd y_2}{y_2}. 
 \end{align*}
By change of variables 
$y_1\to \frac{ m_1|m_2|y_1}{c^2|C|} $ and $y_2 \to \frac{cy_2}{m_1|C|}$, we obtain for $-\Re s_1$ large
\begin{align*}
H_k(s_1,s_2)&=  \sum_{m_1|cm}  \frac {c^2} {c^{3s_1} |m|^{s_1+s_2}} \sum_{m_2\neq 0} \frac{A_F(m_1,m_2)}{m_1^{1-2s_1}|m_2|^{1-s_1}}
S\pra{am,m_2;\frac{cm}{m_1}}\int_0^\infty \int_0^\infty \int_{-\infty}^\infty 
\pra{2\pi i\frac{c^3mx_3y_2}{m_1^2|m_2|y_1}}^k\\
&\quad\quad\quad\quad 
e\pra{ -\frac{\sign m_2}{y_1y_2}\Re\frac{1}{-x_3+i}}
\WJ\pra{
\begin{pmatrix}
\frac{1}{y_1|-x_3+i|} & &\\
&|-x_3+i|y_2&\\
&&1
\end{pmatrix}
}\dd x_3y_1^{s_1}y_2^{s_2}\frac{\dd y_1}{y_1}\frac{\dd y_2}{y_2}
\end{align*}
\begin{align*}
&= \sum_{m_1|cm}  \frac {c^2} {c^{3s_1} |m|^{s_1+s_2}} \sum_{m_2\neq 0} \frac{A_F(m_1,m_2)}{m_1^{1-2s_1}|m_2|^{1-s_1}}
S\pra{am,m_2;\frac{cm}{m_1}}
\pra{2\pi i\frac{c^3m}{m_1^2|m_2|}}^k      \\&
\spaces
 \int_0^\infty \int_0^\infty \int_{-\infty}^\infty
 e\pra{ -\frac{\sign m_2}{y_1y_2}\Re\frac{1}{-x_3+i}}
\\
&\spaces\spaces
\WJ\pra{
\begin{pmatrix}
\frac{1}{y_1|-x_3+i|} & &\\
&|-x_3+i|y_2&\\
&&1
\end{pmatrix}
}y_1^{s_1}y_2^{s_2}
\pra{\frac{x_3y_2}{y_1}}^k
\dd x_3\frac{\dd y_1}{y_1}\frac{\dd y_2}{y_2}
\end{align*}
\begin{align}
&\nonumber = \sum_{m_1|cm}  \frac {c^2} { |m|^{s_2}} \sum_{m_2\neq 0} \frac{A_F(m_1,m_2)}{m_1|m_2|} 
S\pra{am,m_2;\frac{cm}{m_1}}
\pra{\frac{c^3|m|}{m_1^2|m_2|}}^{k-s_1}\pra{2\pi i\sign m}^k\\
&\spaces
 \int_0^\infty \int_0^\infty \int_{-\infty}^\infty
 e\pra{ -\frac{\sign m_2}{y_1y_2}\Re\frac{1}{-x_3+i}}
\label{dirichlet2}
\\
&\spaces\spaces
\WJ\pra{
\begin{pmatrix}
\frac{1}{y_1|-x_3+i|} & &\\
&|-x_3+i|y_2&\\
&&1
\end{pmatrix}
}y_1^{s_1}y_2^{s_2}
\pra{\frac{x_3y_2}{y_1}}^k
\dd x_3\frac{\dd y_1}{y_1}\frac{\dd y_2}{y_2}.\nonumber
\end{align}
Here $\sum\limits_{m_1|cm}$ means that a positive integer $m_1$ varies over the divisors of $cm$, whereas $cm$ may be negative.

\begin{lemma}
For $-\Re s_1\gg 1$, $H_k(s_1,s_2)$ is absolutely convergent for all $s_2\in \mathbb{C}$. 
\end{lemma}

\subsection{Gamma factors}

By \cite[page 161]{bump} we have the Gamma factors 
\begin{align}\nonumber
\int_0^\infty\int_0^\infty&\WJ\pra{
\begin{pmatrix}
y_1y_2 & &\\
&y_1&\\
&&1
\end{pmatrix}
}y_1^{s_1-1}y_2^{s_2-1}\frac{\dd y_1}{y_1} \frac{\dd y_2}{y_2}\\&=\frac{\mathsf C  \pi^{-s_1-s_2}}{4\Gamma\pra{\frac{s_1+s_2}{2}}} \gammafactorminus{s_2}\gammafactor{s_1}.\label{stadeformula1}
\end{align}

Let us recall a famous formula for $k=0, 1$
\nothing{
$$\int_{-\infty}^\infty
 e\pra{ {x_3}y}
(x_3^2+1)^{-\frac{s_1+s_2}{2}}x_3^k\dd x_3=i^k\sign y \frac{2\pi^{\frac{s_1+s_2}{2}}}{\Gamma\pra{\frac{s_1+s_2}{2}}}y^{\frac{s_1+s_2}{2}-\frac{1}{2}}K_{\frac{s_1+s_2}{2}-\frac{1}{2}-k}\pra{2\pi y}.$$
}
$$\int_{-\infty}^\infty e(uy)(u^2+1)^{-s}u^k\dd u
=(i\sign y)^k \frac{2\pi^s|y|^{s-\frac 1 2}}{\Gamma(s)}K_{s-\frac 1 2-k }(2\pi|y|).
$$
Then have the integral in \eqref{dirichlet2}
\begin{align}
\nonumber
& \int_0^\infty \int_0^\infty \int_{-\infty}^\infty
 e\pra{ -\frac{\sign m_2}{y_1y_2}\Re\frac{1}{-x_3+i}}
 \\ &\spaces
\WJ\pra{
\begin{pmatrix}
\frac{1}{y_1|-x_3+i|} & &\\ \nonumber
&|-x_3+i|y_2&\\
&&1
\end{pmatrix}
}
\pra{\frac{x_3y_2}{y_1}}^k
\dd x_3y_1^{s_1}y_2^{s_2}\frac{\dd y_1}{y_1}\frac{\dd y_2}{y_2}\\\nonumber
\end{align}
\begin{align}\nonumber
=\;&(\sign m_2)^k\int_0^\infty \int_0^\infty \int_{-\infty}^\infty
 e\pra{ \frac{x_3}{y_1y_2\pra{x_3^2+1}}}
\\\nonumber
&\spaces
\WJ\pra{
\begin{pmatrix}
\frac{1}{y_1\sqrt{x_3^2+1}} & &\\
&y_2\sqrt{x_3^2+1}&\\
&&1
\end{pmatrix}
}
\pra{\frac{x_3y_2}{y_1}}^k
\dd x_3y_1^{s_1}y_2^{s_2}\frac{\dd y_1}{y_1}\frac{\dd y_2}{y_2}\nonumber
\end{align}
\begin{align}\nonumber
=\;&(\sign m_2)^k\int_0^\infty \int_0^\infty \pra{\int_{-\infty}^\infty
 e\pra{ \frac{x_3}{y_1y_2}}
(x_3^2+1)^{-\frac{s_1+s_2}{2}}x_3^k\dd x_3}
\\\nonumber
&\spaces\spaces
\WJ\pra{
\begin{pmatrix}
\frac{1}{y_1} & &\\
&y_2&\\
&&1
\end{pmatrix}
}y_1^{s_1}y_2^{s_2}
\pra{\frac{y_2}{y_1}}^k
\frac{\dd y_1}{y_1}\frac{\dd y_2}{y_2}
%
%
%
\end{align}
\begin{align}\nonumber
=\;&(\sign m_2)^k\int_0^\infty \int_0^\infty \pra{i^k\frac{2\pi^{\frac{s_1+s_2}{2}}}{\Gamma\pra{\frac{s_1+s_2}{2}}}\pra{\frac{1}{y_1y_2}}^{\frac{s_1+s_2}{2}-\frac{1}{2}}K_{\frac{s_1+s_2}{2}-\frac{1}{2}-k}\pra{\frac{2\pi}{ y_1y_2}} }
\\\nonumber
&\spaces\spaces
\WJ\pra{
\begin{pmatrix}
\frac{1}{y_1} & &\\
&y_2&\\
&&1
\end{pmatrix}
}y_1^{s_1}y_2^{s_2}
\pra{\frac{y_2}{y_1}}^k
\frac{\dd y_1}{y_1}\frac{\dd y_2}{y_2}\\\label{stadeformula2}
=\;& (i\sign m_2)^k\mathsf C  \frac{
\pi^{\frac{s_1+s_2}{2}}}{4\Gamma\pra{\frac{s_1+s_2}{2}}}
\pi^{\frac{-3s_2-6k+3s_1-3}{2}}
\\
&\quad\;\quad \gammafactorminus{s_2} \gammafactorminus{1-s_1+2k}\nonumber
\end{align}
The last equality comes from \cite[page 357]{stade}, as well as  from \cite{bump}.

\subsection{The theorem}
\begin{theorem}\label{main}
Let $F$ be a Maass form on $\GL 3$ for $\Gamma_0(N)$ with a Dirichlet character $\psi: \mathbb{ Z}/N\mathbb Z\to \mathbb C^{\times}$ as defined in Section \ref{section_background}. It has the Fourier-Whittaker coefficient $A_F(m_1,m_2)$.
Let $c$ be a positive integer with $N|c$ and let $a$, $\bar a$ be two integers with $a\bar a\equiv 1\; \mod c$. Let $m\neq 0$ be an integer.
We have the identity
\begin{align}
\sum_{n=1}^\infty &\frac{A_F(n,m)}{n^{s}}e\pra{-\frac{n\bar a}{c}}
=c\pi^{-\frac{3}{2}+3s}\psi(a)\sum_{m_2\neq 0}\sum_{m_1|cm}\textsf G(s)\frac{A_F(m_1,m_2)}{m_1|m_2|} S\pra{am,m_2;\tfrac{cm}{m_1}}\pra{\frac{c^3|m|}{m_1^2|m_2|}}^{-s},
\label{equation_voronoi}
\end{align}
where we define for abbreviation
$$\textsf  G(s):=\frac{1}{2}\pra{\frac{\gammafactorminus{1-s}}{\gammafactor{s}}+i\sign (mm_2)\frac{\gammafactorminus{2-s}}{\gammafactor{s+1}}}.$$
The Dirichlet series on the left is convergent when $\Re s$ is large; the Dirichlet series on the right is convergent when $-\Re s$ is large; both have analytic continuation to the whole complex plane. 
Here $S(*,*;*)$ is the classical Kloosterman sum and the unordered triple $(\alpha, \beta, \gamma)$ is determined by 
the eigenvalues of $F$ under Casimir operators.
\end{theorem}

\begin{proof}[Proof of Theorem \ref{main}]
The holomorphic continuation of $H_k(s_1,s_2)$ in $s_1$ comes from  Riemann's method of
dividing the integration of $y_1$ on $(0,\infty)$ into  $(0,1]$ and $[1,\infty)$,
and then making the transformation $y_1 \to 1/y_1$ on  $(0,1]$.
Re-examining the two ways to express $H_k(s_1+k,s_2)$ (\eqref{dirichlet1} and \eqref{dirichlet2}) along with the explicit Gamma factors \eqref{stadeformula1} and \eqref{stadeformula2}, we can cancel all parts involving $s_2$ and we obtain
\begin{align*}
\psi(d)\sum_{m_1=1}^\infty & 
\frac{A_F(m_1,m)}{m_1^{s}}\pra{e\pra{-\frac{m_1d}{c}}+(-1)^k e\pra{\frac{m_1d}{c}}}\\
&=c\pi^{-\frac{3}{2}+3s}\textsf G_k(s)\sum_{m_2\neq 0}\sum_{m_1|cm}(i\sign (mm_2))^k\frac{A_F(m_1,m_2)}{m_1|m_2|} S\pra{am,m_2;\frac{cm}{m_1}}\pra{\frac{c^3|m|}{m_1^2|m_2|}}^{-s},
\end{align*}
where we define for abbreviation
$$\textsf  G_k(s):=\frac{\gammafactorminus{1+k-s}}{\gammafactor{s+k}}.$$
The left side is absolutely convergent when $\Re s$ large and the Dirichlet series on the right side is absolutely convergent when $-\Re s$ large. They have analytic continuation to the whole complex plane and satisfy the previous functional equation.
Combining the two cases $k=0$ and $1$, we finish the proof of Theorem \ref{main}.
\end{proof}

\section{Application}\label{sectionapplication}
\nothing{
Let $\pi_p$ be a representation of $\GL 2(\mathbb{Q}_p)$.  Let $\chi_p$  be a representation of $\GL 1 (\mathbb{Q}_p)$.
In the theory of automorphic representation, it is a natural question to consider the $L$-factor
$L(s, \pi_p\times\chi_p),$
whose inverse is a polynomial in $\mathbb C[p^{-s}]$.
It is sometimes known that $$L(s, \pi_p\times\chi_p)\equiv 1$$ when $\chi_p$ is of sufficiently large conductor. 
This is also related to the stability of the local gamma factors at non-archimedean places (see \cite{jacquetshalika}).
}

As an application of Theorem \ref{main}, we will produce a result on $\GL 3$ 
analogous to the stability of the local gamma factors at non-archimedean places, showing that 
the global $L$-function $L(s,F\times \chi)$ 
and its functional equation
do not contain $p$-factor for $p\nmid N$, 
whenever a Dirichlet character $\chi$ is ramified enough at places $p\mid N$.

Let $\chi $ be a primitive character of modulo $c$ and let $c$ be a multiple of $N$.
Define the partial $L$-functions
$$L\levelN(s,F'\times \bar \psi \bar \chi):=\sum_{\underset{\scriptstyle(n,N)=1}{n=1}}^\infty  \frac{A_F(n,1)\bar\psi\bar\chi(n)}{n^s}$$
and 
$$L\levelN(s,F\times \chi):=\sum_{\underset{\scriptstyle(n,N)=1}{n=1}}^\infty  \frac{A_F(1,n)\chi(n)}{n^s}.$$


\begin{theorem}
Let $F$ be a Maass form on $\GL 3$ for $\Gamma_0(N)$ with a Dirichlet character $\psi: \mathbb{ Z}/N\mathbb Z\to \mathbb C^{\times}$ as defined in Section \ref{section_background}. 
Let $\chi$ be a primitive Dirichlet character modulo $c$, which is a multiple of $N$. 
Let us assume that the character $\psi\chi$ is a primitive character modulo $c$, i.e., the conductor of $\psi\chi$ is $c$. 
Then we have the functional equation
$$L\levelN(s,F\times \chi)=\Xi(s)L\levelN(1-s,F'\times \bar \psi \bar \chi)$$
where $\Xi(s)$ stands for 
$$\Xi(s)=\tau(\psi\chi)\tau(\chi)^2c^{-3s}\cdot i^\kappa\pi^{3(s-\tfrac{1} 2)}
\frac{\gammafactor{1-s+\kappa}}{\gammafactorminus{s+\kappa}}
$$
and let $\kappa=0$ when $\psi\chi(-1)=1$ and $\kappa=1$ when $\psi\chi(-1)=-1$. 
\end{theorem}

\begin{proof}
Multiply $\overline{\chi(a)}\overline{\psi(a)}$ on both sides of the Voronoi formula 
\eqref{equation_voronoi}
in Theorem \ref{main} and sum (with respect to $a$) over the reduced residue system modulo $c$. 
We have by \cite[Lemma 3.4]{kiralzhou}
$$\sum_{\underset {\scriptstyle (a,c)=1}{a\; \mod c}}\overline{\chi(a)}
 S(am,m_2;\tfrac{cm}{m_1})
=\begin{cases}g(\bar \chi,c,m_1)g(\bar \chi,\tfrac{cm}{m_1},m_2), \quad &\text{ if } m_1|m,\\
0, \quad&\text{ otherwise.}
\end{cases}
$$
Let $m$ be $1$, which forces $m_1$ to be $1$ and now we have 
$$\sum_{\underset {\scriptstyle (a,c)=1}{a\; \mod c}}
\overline{\chi(a)} S(a,m_2;\tfrac{c}{m_1})=\delta_{m_1=1} \tau(\bar \chi)  g(\bar \chi,c,m_2).$$
We have by \cite[Lemma 2.3]{kiralzhou}
$$g(\bar\chi,c,m_2)=\begin{cases}
\tau(\bar \chi)\chi(m_2),\quad &\text{if }(m_2,c)=1,\\
0,\quad &\text{otherwise.}
\end{cases}
$$
Recall Proposition \ref{even_odd} we have the right side of \eqref{equation_voronoi} 

\begin{align*}
&
c\pi^{-\frac{3}{2}+3s}\sum_{\underset {\scriptstyle (a,c)=1}{a\; \mod c}}\overline{\chi(a)}\sum_{m_2\neq 0}\sum_{m_1|c}\textsf G(s)\frac{A_F(m_1,m_2)}{m_1|m_2|} S\pra{a,m_2;\frac{c}{m_1}}\pra{\frac{c^3}{m_1^2|m_2|}}^{-s}
\\
=\;
&
c^{1-3s}\pi^{-\frac{3}{2}+3s} \tau(\chib)^2
\sum_{m_2\neq 0}\textsf G(s)\frac{A_F(1,m_2)}{|m_2|^{1-s}} \chi(m_2)
\\
=\;&
c^{1-3s}\pi^{-\frac{3}{2}+3s} \tau(\chib)^2
i^\kappa\mathsf G_\kappa(s)
\sum_{\underset{\scriptstyle (m_2,c)=1}{m_2=1}}^{\infty}\frac{A_F(1,m_2)\chi(m_2)}{m_2^{1-s}}. 
\end{align*}
From the left side of \eqref{equation_voronoi}, we have
\begin{align*}
&\sum_{\underset {\scriptstyle (a,c)=1}{a\; \mod c}}\overline{\chi(a)}\overline{\psi(a)}\sum_{n=1}^\infty \frac{A_F(n,1)}{n^{s}}\exp\pra{-2\pi i\frac{n\bar a}{c}}
\\
=\;&\sum_{n=1}^\infty \frac{A_F(n,1)}{n^{s}}g(\psi\chi,c,-n)
\\
=\;&(-1)^\kappa
\tau(\psi\chi)\sum_{\underset{\scriptstyle (n,c)=1}{n=1}}^\infty \frac{A_F(n,1)}{n^{s}}\bar\psi\bar\chi(n).
\end{align*}
Hence the theorem is proved after changing $s$ to $1-s$.
\end{proof}

Let us assume that $F$ is an eigenfunction under $T_n$ and $T_n^*$ for all $(n, N)=1$.
Let us assume $A_F(1,1)=1$
In such a case, by Section \ref{subsection_hecke_operator}, we have 
$$A_F(n,1)\bar\psi\bar\chi(n)=A_{\tilde F}(1,n)\bar\chi(n)$$
and 
$$L\levelN(s,F'\times \bar \psi \bar \chi)=
L\levelN(s,\tilde F\times \bar \chi):=
\sum_{\underset{\scriptstyle(n,N)=1}{n=1}}^\infty  \frac{A_{\tilde F}(1,n)\bar\chi(n)}{n^s}.$$
Moreover we have
$$L\levelN(s,\tilde F\times\bar \chi)=\overline{L\levelN(\bar s, F\times \chi)}.$$
Actually 
$L\levelN(s,F\times \chi)$
has the Euler product by the Hecke relations
$$L\levelN(s,F\times\chi)=\prod_{p\nmid N} \pra{1-\frac{A_F(1,p)\chi(p)}{p^s}+\frac{A_F(p,1)\chi(p)^2}{p^{2s}}-\frac{\chi^3\psi(p)}{p^{3s}}}^{-1}.$$

\begin{theorem}\label{theorem_L_function}
Let $F$ be a Maass form on $\GL 3$ for $\Gamma_0(N)$ with a Dirichlet character $\psi: \mathbb{ Z}/N\mathbb Z\to \mathbb C^{\times}$ as defined in Section \ref{section_background}. 
Let us assume that $F$ is an eigenfunction under $T_n$ and $T^*_n$ for all $n$ with $(n,N)=1$, as described in Section \ref{subsection_hecke_operator} as well as $A_F(1,1)=1$. 
Let $\chi$ be a primitive Dirichlet character modulo $c$, which is a multiple of $N$. 
Let us assume that the character $\psi\chi$ is a primitive character modulo $c$, i.e., the conductor of $\psi\chi$ is $c$. 
Then we have the functional equation
$$L\levelN(s,F\times \chi)=\Xi(s)L\levelN(1-s,\tilde F\times  \bar \chi)$$
where $\Xi(s)$ stands for 
$$\Xi(s)=\tau(\psi\chi)\tau(\chi)^2c^{-3s}\cdot i^\kappa\pi^{3(s-\tfrac{1} 2)}
\frac{\gammafactor{1-s+\kappa}}{\gammafactorminus{s+\kappa}}
$$
and let $\kappa=0$ when $\psi\chi(-1)=1$ and $\kappa=1$ when $\psi\chi(-1)=-1$. 
\end{theorem}

\begin{remark}
The aforementioned functional equation would be the same as that derived from the theory of automorphic representation.
In Theorem \ref{theorem_L_function}, the $L$-function $L\levelN(s,\pi\times \chi)$ does not include any places $p\mid N$. 
\end{remark}

\nothing{
\subsection{Newform theory}

When $\psi$ is a primitive character of modulo $N$, we would like to prove that any Hecke-Maass form $F$ for $\Gamma_0(N)$ with $\psi$ is always a ``newform'' with $A_F(1,1)\neq 0$.

We start with two fundamental functional equations

\begin{align}
\sum_{n=1}^\infty &\frac{A_F(n,m)}{n^{s}}
g(1,lN,n)
=c\pi^{-\frac{3}{2}+3s}\sum_{m_2\neq 0}\sum_{m_1|lm}\textsf G(s)\frac{A_F(m_1,m_2)}{m_1|m_2|} 
g(\psi,lN,m_1)g(\psi,N\tfrac{lm}{m_1},m_2)
\pra{\frac{c^3|m|}{m_1^2|m_2|}}^{-s},
\label{newform1}
\end{align}

\begin{align}
\sum_{n=1}^\infty &\frac{A_F(n,m)}{n^{s}}
g(\psi,lN,n)
=c\pi^{-\frac{3}{2}+3s}\sum_{m_2\neq 0}\sum_{m_1|lNm}\textsf G(s)\frac{A_F(m_1,m_2)}{m_1|m_2|} 
g(1,lN,m_1)g(1,\tfrac{lNm}{m_1},m_2)
\pra{\frac{c^3|m|}{m_1^2|m_2|}}^{-s},
\label{newform2}
\end{align}

In \eqref{newform2}, let $m=1$ and $(l,N)=1$. 
$g(\psi,lN,n)=0$ for $(n,N)\neq 1$. 
The left side of \eqref{newform2} is identically zero. 
For $(m_1,N)=1$, 
$g(1,m_1N,m_1)g(1,N,m_2)\neq 0$.
Thus we have $A_F(m_1,m_2)=0$ for $(m_1,N)=1$.

In \eqref{newform1}, let $m=1$ and $l=1$.
This forces the right side of \eqref{newform1} to $m_1=1$. Hence the right side of \eqref{newform1} is identically zero.
Because of $g(1,N,n)\neq 0$, we have $A_F(n,1)$ for any $n$.

In \eqref{newform2}, let $m=1$.  The left side of \eqref{newform2} is identically zero.
Let $l=m_1$. 
$g(1,lN,m_1)g(1,\tfrac{lN}{m_1},m_2)\neq 0$.
Hence $A_F(m_1,m_2)=0$ for any $m_1, m_2$. 

XXXXXXX

XXXXXXX

XXXXXXX

Because of $A_F(n,1)=0$ for $(n,N)=1$
and 

Let us assume $m=1$ first and $\psi=\chib$. 

Let us assume $c=lN$ with $(l,N)=1$

The left side is zero.   

Let $l=m_1$.
The Gauss sum on the right is not zero.

Hence $A_F(m_1,m_2)=0$ for $(m_1,N)=1$. 
\\

Now let $\chi=\mathbf 1$ and $m=1$ and $l=1$.

This forces $(m_1,N)=1$.

Hence the right side is zero.

Hence $A_F(n,1)=0$ for any $n$.
\\

Now let $\chi=\psib$. 

Let $m=1$.

$c=Nl$. 

Let $l=m_1$.

The Gauss sum on the right is nonzero.

Hence $A_F(m_1,m_2)=0$.

}

\section{The second scenario}\label{sectionsecondscenario}
In this section, we will obtain a Voronoi formula on $\GL 3$ in the second scenario, where the level of the automorphic form and the conductor of the additive twist are coprime.  
Let $F$ be an automorphic form for $\Gamma_0(N)$ and 
a Dirichlet character $\psi :\mathbb{Z}/N\mathbb{ Z}\to \mathbb{C}^\times$. 
Let  $A_F(m_1,m_2)$  be the Fourier-Whittaker coefficient of $F$ for $m_1,m_2>0$. 
Let $\tilde{F}$ be the contragredient form of $F$ with the Fourier-Whittaker coefficient $A_{\tilde{F}}(m_1,m_2)$ for $m_1,m_2>0$. We have for $(mn,N)=1$ that $$A_F(m,n)=\psi(m)\psi(n)\ATF(n,m).$$
We assume for $(n_1n_2,N)=1$ the Hecke relations 
\begin{equation}\label{eq_double_hecke1}
\AF(1,m)\AF(n_2, n_1)=\sum_{abc=m, b|n_1, c|n_2}\AF(\tfrac{bn_2}c,\tfrac{ an_1}b)\psi (c),
\end{equation}
\begin{equation}\label{eq_double_hecke2}
\ATF(1,m) \ATF(n_2, n_1)=\sum_{abc=m, b|n_1, c|n_2} \ATF(\tfrac{bn_2}c,\tfrac{ an_1}b)\bar\psi (c)
\end{equation}
are satisfied.

Let $\chis:\mathbb{Z}/\cs\mathbb{Z}\to \mathbb{C}^\times$ be a primitive Dirichlet character modulo $\cs$. 
For $(\cs,N)=1$, let
$$L(s,F\times \chis)=\sum\limits_{n=1}^\infty \frac{\AF(1,n)\chis (n)}{n^s}$$ be the $L$-function of $F$ twisted by $\chis$. 
It is well known that they satisfy a functional equation
\begin{equation}\label{functionalequation}
L(s,F\times \chis)=G_\pm (s)\psi(\cs)\chis(N) \cs^{-3s}\tau(\chis)^3 L(1-s, \tilde F\times \chibs),
\end{equation}
where $\tau(\chis)$ is the Gauss sum of $\chis$, $G_\pm$ equals $G_+$ when $\chis(-1)=1$ and 
$G_\pm$ equals $G_-$  when $\chis(-1)=-1$.
More precisely, if $F$ is a Maass form as described in Section \ref{section_background} we have 
$$G_\pm(s)=i^k\tau(\psi)\epsilon(F){N}^{\tfrac{1}{2}-s}\pi^{3(s-\tfrac 1 2)}\frac{\gammafactorminus{1-s+k}}{\gammafactor{s+k}}$$
with $k=0$ for $G_+$ and $k=1$ for $G_-$ if $\psi(-1)=1$
and $k=1$ for $G_+$ and $k=0$ for $G_-$ if $\psi(-1)=-1$, as well as $|\epsilon(F)|=1$.


We will use the method developed in \cite{kiralzhou} to prove the following theorem.

\begin{theorem}\label{voronoi2}
Let us assume that $F$ (respectively $\tilde{ F}$) is associated with numbers $\AF(m_1,m_2)$ (respectively $\ATF(m_2,m_1)$) for positive integers $m_1$, $ m_2$, with $(m_1,N)=1$.
Let us assume that  $\AF(m_1,m_2)$ and  $\ATF(m_2,m_1)$ satisfy the Hecke relations \eqref{eq_double_hecke1},
\eqref{eq_double_hecke2}.
Let $c$ be a positive integer with $(c,N)=1$.
Let $q$ be a positive integer with $(q,N)=1$. 
Let us assume \eqref{functionalequation} for any primitive Dirichlet character $\chis$ modulo $\cs$ with $(\cs,N)=1$.
Let $a$ be an integer with $(a,c)=1$. Let $\bar a$ (respectively $\bar N$) be the multiplicative inverse of $a$ (respectively $N$) modulo $c$. 
We have the following formula 
\begin{align}\nonumber
\sum_{n=1}^\infty &\frac{\AF(q,n)}{n^{s}}e\pra{\frac{\bar a n}{c}}\\&=c\psi(cq) \sum_\pm\sum_{m_2=1}^\infty\sum_{m_1|cq}
\frac{G_+(s)\mp G_-(s)}2
\frac{\ATF(m_1,m_2)}{m_1m_2} S\pra{\pm\bar N q a,m_2;\tfrac{cq}{m_1}}\pra{\frac{m_1^2m_2}{c^3q}}^{s}.\label{voronoi2formula}
\end{align}
The Dirichlet series on the left is convergent when $\Re s$ is large; the Dirichlet series on the right is convergent when $-\Re s$ is large; both have analytic continuation to the whole complex plane. 
\end{theorem}
Define the Gauss sum $$g(\chi^*,c,m):=\sum\limits_{\underset{\scriptstyle (u,c)=1}{u=1}}^c e\pra{\frac{um}{c}}\chis(u),$$
which is actually the Gauss sum of $\chi$. 
Define the divisor function 
$$\sigma_{s}^{(N)}(m,\chi):=\sum_{d|m,(d,N)=1}\chi(\frac m d)(\frac m d)^{s}.$$
\begin{lemma}\label{ramanujan}
	 Let $\Re(s)>1$. Define a Dirichlet series
 	 $$I^{(N)}(s, \chi^*,c^*,m)=\sum_{(\ell,N)=1}\frac{g(\chi^*, \ell c^*,m )}{\ell ^s}$$
 	 as a generating function for the nonprimitive Gauss sums induced from $\chi^*$.
 	 It satisfies the identity
 	 $$\tau(\chi^*)m^{1-s}\sigma_{s-1}^{(N)}(m,\overline{\chi^*})L^{(N)}(s,\chi^*)^{-1}=I^{(N)}(s,\chi^*,c^*,m).$$
\end{lemma} 

\begin{proof}
Mutatis mutandis, this is \cite[Lemma 2.4]{kiralzhou}. 
\end{proof}

We define two Dirichlet series for $\Re(s)\gg 1$
$$H(q,\tfrac{c}\cs,\chis,s):=\sum_n \frac{\AF(q,n)g(\chibs,c,n)}{n^s(c/c^*)^{1-2s}}$$
and for $\Re(1-s)\gg 1$
$$G(q,\tfrac c \cs,\chis,s):=\frac{\chis(-N)\psi(qc) G_\pm(s)}{c^{3s-1}(c/\cs)^{1-2s}}
\sum_{d\cs|qc}\sum_n \frac{\ATF(d,n)}{dn} \pra{\tfrac{d^2n}{q}}^s
g(\chis,c,d)g(\chis,\tfrac{qc}{d},n).
$$

\nothing{
Using the method developed in \cite[Section 4]{goldfeldli1} and \cite{zhou} we have the following theorem.
\begin{theorem}
Let $c$ be a prime number and $(c,N)=1$. Let $a$ be an integer with $(a,c)=1$. Let $\bar a$ be the multiplicative inverse of $a$ $\mod c$. We have the following formula 
\begin{align*}
\sum_{n=1}^\infty &\frac{\AF(1,n)}{n^{s}}e\pra{-\frac{a n}{c}}\\&=c\psi(c) \sum_\pm\sum_{m_2=1}^\infty\sum_{m_1|c}
\frac{G_+(s)\mp G_-(s)}2
\frac{\ATF(m_1,m_2)}{m_1m_2} S\pra{\pm\bar N \bar a,m_2;\frac{c}{m_1}}\pra{\frac{c^3}{m_1^2m_2}}^{-s}.
\end{align*}
\end{theorem}
The method in \cite{kiralzhou} may produce a more general formula for 
$$\sum_n \AF(m,n)e\pra{\frac{an}c}$$
when $(m,N)=1$ and $c$ is not necessarily a prime. 
}

\subsection{Double Dirichlet series}
In the following calculation, $\sum\limits_{x}$ stands for $\sum\limits_{x=1}^\infty$ 
and $\sum\limits_{(x,N)=1}$ for $\sum\limits_{\underset{\scriptstyle (x,N)=1}{x=1}}^\infty$.
\begin{theorem}\label{convolution}
For positive integers $q,m$ which are coprime with $N$, we define for $\Re (s)\gg 1$
$$\mathbf{H}(q,m,\chis, s):=\sum_{d_2|q}\sum_{d_1\ell=m} 
\frac { \psi(d_2)  \chis(d_1d_2)} { d_2^s} 
H(\tfrac{qd_1}{d_2},\ell c^*,\chis,s)$$
and for $\Re(1-s)\gg 1$
$$
\mathbf{G}(q,m,\chis, s):=
\sum_{d_2|q}\sum_{d_1\ell=m} 
\frac { \psi(d_2)  \chis(d_1d_2)} { d_2^s} 
G(\tfrac{qd_1}{d_2},\ell c^*,\chis,s).$$
Both have analytic continuation to all $s\in \mathbb C$ and 
$$\mathbf{H}(q,m,\chis, s)=\mathbf{G}(q,m,\chis, s).$$
\end{theorem}

\begin{proof}
We define a double Dirichlet series
\begin{equation*}\label{doubledirichlet}
Z(s,w):=\frac{L_q\levelN(2w-s,F) L(s,F\times\chi^*)}{L\levelN(2w-2s+1,\chibs)},
\end{equation*}
where $L_q\levelN(s,F) :=\sum\limits_{(n,N)=1} \frac{\AF(q,n)}{n^s} $ is a Dirichlet series with $(q,N)=1$. 
Using the Hecke relation we have 
\begin{align*}
Z(s,w)&=L\levelN(2w-2s+1,\chibs)^{-1}\sum_{(n,N)=1}\sum_m
\frac {A\pra{q,n} \AF(1,m) \chis(m)} {n^{2w-s}m^s}\\
&=L\levelN(2w-2s+1,\chibs)^{-1}\sum_{(n,N)=1}\sum_m
\sum_{\underset{\scriptstyle d_1|n, d_2|q}{d_0d_1d_2=m}}
\frac {A\pra{\frac{qd_1}{d_2}, \frac{nd_0}{d_1}} \psi(d_2) \chis(m)} {n^{2w-s}m^s}.
\end{align*}
By a change of variable $n/d_1\to n$ we have
\begin{align*}
Z(s,w)&=
L\levelN(2w-2s+1,\chibs)^{-1}\sum_{(n,N)=1}\sum_m
\sum_{\underset{\scriptstyle (d_1,N)=1, d_2|q}{d_0d_1d_2=m}}
\frac {A\pra{\frac{qd_1}{d_2}, {nd_0}}  \psi(d_2)  \chis(m)} {(nd_1)^{2w-s}m^s}
\\
&=L\levelN(2w-2s+1,\chibs)^{-1}\sum_{(n,N)=1}\sum_{(d_1,N)=1}\sum_{d_0}\sum_{d_2|q} 
\frac {A\pra{\frac{qd_1}{d_2}, {nd_0}} \psi(d_2)  \chis(d_0d_1d_2)} {n^{2w-s} d_1^{2w}(d_0d_2)^s}
\\
&=L\levelN(2w-2s+1,\chibs)^{-1}\sum_{n}\sum_{(d_1,N)=1} \sum_{d_2|q}
\frac {A\pra{\frac{qd_1}{d_2}, {n}} \psi(d_2)  \chis(d_1d_2)} {n^{2w-s} d_1^{2w}d_2^s} \sigma\levelN_{2w-2s}(n,\chis)
\\
&=L\levelN(2w-2s+1,\chibs)^{-1}\sum_{n}\sum_{(d_1,N)=1} \sum_{d_2|q}
\frac {A\pra{\frac{qd_1}{d_2}, {n}} \psi(d_2)  \chis(d_1d_2)} {n^s d_1^{2w}d_2^s} \frac{\sigma\levelN_{2w-2s}(n,\chis)}{n^{2w-2s}}
\end{align*}
Applying Lemma \ref{ramanujan} we have
\begin{align}\nonumber
Z(s,w)&=\tau(\chibs)^{-1}\sum_{n}\sum_{d_2|q} \sum_{(d_1,N)=1} 
\frac {A\pra{\frac{qd_1}{d_2}, {n}} \psi(d_2)  \chis(d_1d_2)} {n^s d_1^{2w}d_2^s} \sum_{(\ell,N)=1}\frac{g(\chibs,\ell c^*, n)}{\ell^{2w-2s+1}}
\\
&=\sum_{(d_1,N)=1}\frac{1}{d_1^{2w}}\pra{\tau(\chibs)^{-1}\sum_{d_2|q}  \sum_{(\ell,N)=1}
\frac { \psi(d_2)  \chis(d_1d_2)} { \ell^{2w}d_2^s} 
H(\tfrac{qd_1}{d_2},\ell c^*,\chis,s)}
.\label{doubledirichletleft}
\end{align}

%
%
%
%



Apply the functional equation \eqref{functionalequation} to the double Dirichlet series $Z(s,w)$
we have
\begin{align*}
Z(s,w)&=G_\pm (s)\psi(\cs)\chis(N) \cs^{-3s}\tau(\chis)^3\frac{L_q\levelN(2w-s,F)  L(1-s, \tilde F\times \chibs)}{L\levelN(2w-2s+1,\chibs)}
\\
&=\frac{G_\pm (s)\psi(\cs)\chis(N) \cs^{-3s}\tau(\chis)^3}{L\levelN(2w-2s+1,\chibs)  }
\sum_{(n,N)=1}\sum_m \frac{\AF(q,n)\ATF (1,m)\chibs(m)}{n^{2w-s}m^{1-s}}
\\
&=\frac{G_\pm (s)\psi(\cs)\chis(N) \cs^{-3s}\tau(\chis)^3}{L\levelN(2w-2s+1,\chibs)  }
\sum_{(n,N)=1}\sum_m \frac{\ATF(n,q)\psi(nq)\ATF (1,m)\chibs(m)}{n^{2w-s}m^{1-s}}
.
\end{align*}
Applying the Hecke relation we have 
$$Z(s,w)=\frac{G_\pm (s)\psi(\cs)\chis(N) \cs^{-3s}\tau(\chis)^3}{L\levelN(2w-2s+1,\chibs)  }
\sum_{(n,N)=1}\sum_m 
\sum_{\underset{\scriptstyle d_1|q, d_2|n}{d_0d_1d_2=m} }
\frac{\ATF(\frac{nd_1}{d_2},\frac{qd_0}{d_1})
\psib(d_2)
\psi(nq)\chibs(m)}{n^{2w-s}m^{1-s}}
$$
By a change of variable $n/d_2\to n$
\begin{align}\nonumber
Z(s,w)&=\frac{G_\pm (s)\psi(\cs)\chis(N) \cs^{-3s}\tau(\chis)^3}{L\levelN(2w-2s+1,\chibs)  }
\sum_{(n,N)=1}
\sum_{d_1|q }\sum_{(d_2,N)=1}
\sum_{d_0}
\frac{\ATF(nd_1,\frac{qd_0}{d_1})
\psi(nq)\chibs(d_0d_1d_2)}{(nd_2)^{2w-s}(d_0d_1d_2)^{1-s}}
\\
&=
{G_\pm (s)\psi(\cs)\chis(N) \cs^{-3s}\tau(\chis)^3}
\sum_{(n,N)=1}
\sum_{d_1|q }
\sum_{d_0}
\frac{\ATF(nd_1,\frac{qd_0}{d_1})
\psi(nq)\chibs(d_0d_1)}{n^{2w-s}(d_0d_1)^{1-s}}
.\label{zswfirstexpression}
\end{align}


We would like to prove that \eqref{zswfirstexpression} equals 
\begin{equation}
\sum_{(d_1,N)=1}\frac{1}{d_1^{2w}}\pra{\tau(\chibs)^{-1}\sum_{d_2|q}  \sum_{(\ell,N)=1}
\frac { \psi(d_2)  \chis(d_1d_2)} { \ell^{2w}d_2^s} 
G(\tfrac{qd_1}{d_2},\ell c^*,\chis,s)}.\label{sumofG}
\end{equation}
We have \eqref{sumofG} equals
$$\tfrac{ \chis(-N)}{\tau(\chibs)}\sum_{d_2|q} \sum_{(d_1,N)=1} \sum_{(\ell,N)=1}
\sum_{d|\tfrac{qd_1}{d_2}\ell}\sum_n
\frac { 
  \chis(d_1d_2)
\psi(qd_1\ell \cs) G_\pm(s)  
  } { (\ell d_1)^{2w+s}\cs^{3s-1}q^sn^{1-s}d^{1-2s}} 
\ATF(d,n) 
g(\chis,\ell \cs,d)g(\chis,\tfrac{qd_1\ell \cs}{d_2d},n).
$$
By applying \cite[Lemma 2.5]{kiralzhou} we have \eqref{sumofG} equals 
$$\tfrac{ \chis(-N)}{\tau(\chibs)}\sum_{d_2|q} \sum_{(m,N)=1}
\sum_{d|\tfrac{qm}{d_2}}\sum_n
\frac { 
  \chis(d_2)
\psi(qm \cs) G_\pm(s)  
  } { m^{2w+s-1}\cs^{3s-1}q^sn^{1-s}d^{1-2s}} 
\ATF(d,n) 
g(\chis,\tfrac{qm \cs}{d_2d},n)\tau(\chis)\chibs(\tfrac d m) \delta_{m|d}.
$$
Taking $f=\tfrac d m$ we have \eqref{sumofG} equals
$$\tfrac{ \chis(-N)\tau(\chis)}{\tau(\chibs)}
 \sum_{(m,N)=1}
\sum_{f|q}
\sum_{d_2|\tfrac q f}
\sum_n
\frac { 
  \chis(d_2)
\psi(qm \cs) G_\pm(s)  
  } { m^{2w-s}\cs^{3s-1}q^sn^{1-s}f^{1-2s}} 
\ATF(fm,n) 
g(\chis,\tfrac{q \cs}{d_2f},n)\chibs(f) 
.
$$
Applying \cite[Lemma 2.5]{kiralzhou} again we have that \eqref{sumofG} equals
\begin{align*}
&\tfrac{ \chis(-N)\tau(\chis)^2}{\tau(\chibs)}
 \sum_{(m,N)=1}
\sum_{f|q}
\sum_n
\frac { 
\psi(qm \cs) G_\pm(s)  
  } { m^{2w-s}\cs^{3s-1}q^sn^{1-s}f^{1-2s}} 
\ATF(fm,n) 
\chibs(f) \chibs(\tfrac{nf}{q})\frac q f \delta_{\frac q f|n}
\\
=\;&\tfrac{ \chis(-N)\tau(\chis)^2}{\tau(\chibs)}
 \sum_{(m,N)=1}
\sum_{f|q}
\sum_n
\frac { 
\psi(qm \cs) G_\pm(s)  
  } { m^{2w-s}\cs^{3s-1} n^{1-s}f^{1-s}} 
\ATF(fm,n\tfrac q f) 
\chibs(f) \chibs(n),
\end{align*}
after a change of variable $nf/q \to n$,
which is identical to \eqref{zswfirstexpression}.

Comparing \eqref{doubledirichletleft} and \eqref{sumofG} and applying the uniqueness theorem for Dirichlet series (in terms of $w$), we finish the proof of the theorem. 
\end{proof}

\begin{proof}[Proof of Theorem \ref{voronoi2}]
We have 
$$H(q,m,\chis, s)=\sum_{e_0|m}\sum_{e_1|qe_0}
\frac{\mu(e_0)\mu(e_1)\chis(e_0e_1)\psi(e_1)}{e_1^s}
\mathbf{H}(\tfrac{qe_0}{e_1},\tfrac{m}{e_0},\chis, s)$$
and 
$$G(q,m,\chis, s)=\sum_{e_0|m}\sum_{e_1|qe_0}
\frac{\mu(e_0)\mu(e_1)\chis(e_0e_1)\psi(e_1)}{e_1^s}
\mathbf{G}(\tfrac{qe_0}{e_1},\tfrac{m}{e_0},\chis, s).$$
By Theorem \ref{convolution}, we have 
$$H(q,\ell,\chis,s)=G(q,\ell,\chis,s).$$
For a Dirichlet character $\chi:\mathbb{Z}/c\mathbb{Z}\to\mathbb{C}^\times$, which is induced from a primitive Dirichlet character $\chis:\mathbb{Z}/\cs\mathbb{Z}\to\mathbb{C}^\times$ with some integer $\cs |c$.
Multiply both side of \eqref{voronoi2formula} by $\chi(a) $ and sum over reduced residue classes modulo $c$
we get $$H(q,\tfrac c \cs,\chis,s)=G(q,\tfrac c \cs,\chis,s)$$  by \cite[Lemma 3.4]{kiralzhou}. Using the orthogonality relation for Dirichlet characters, we have that $H(q,\tfrac c \cs,\chis,s)=G(q,\tfrac c \cs,\chis,s)$ for all $\chi \;\mod c$ is equivalent to \eqref{voronoi2formula} for all $a\;\mod c$ with $(a,c)=1$. 
\end{proof}


\vspace*{\fill}
\noindent {\scshape{Fan Zhou}} \\
\small{
fan.zhou@maine.edu\\
\\
Department of Mathematics and Statistics\\
The University of Maine\\
5752 Neville Hall, Room 333\\
Orono, ME 04469, USA\\
}

\end{document}